\documentclass[11pt,a4paper,reqno]{amsart}

\usepackage{amssymb, amsthm, amsmath}
\usepackage[hdivide={2.5cm,,2.5cm}, vdivide={3.5cm,,2.8cm}]{geometry}

\newtheorem{thm}{Theorem}[section]
\newtheorem{cor}[thm]{Corollary}
\newtheorem{lem}[thm]{Lemma}

\newtheorem{knownthm}{Theorem}[section]

\newtheorem{knownlem}[knownthm]{Lemma}

\theoremstyle{definition}

\newtheorem{knowndefinition}[knownthm]{Definition}
\newtheorem{rem}[thm]{Remark}

\newcommand{\CC}{\overline{\mathbb{C}}}
\newcommand{\C}{\mathbb{C}}
\newcommand{\D}{\mathbb{D}}
\renewcommand{\H}{\mathbb{H}}

\newcommand{\N}{\mathbb{N}}
\newcommand{\R}{\mathbb{R}}

\renewcommand{\SS}{\mathcal{S}}

\newcommand{\no}{\noindent}
\newcommand{\dstyle}{\displaystyle}
\def\Re{{\sf Re}\,}

\newcommand{\closure}{\overline}
\newcommand{\vareps}{\varepsilon}
\newcommand{\de}{\partial}

\newcommand{\hol}{\textup{Hol}}
\newcommand{\LC}{\mathtt{LC}}
\newcommand{\EF}{\mathtt{EF}}
\newcommand{\HV}{\mathtt{HV}}
\newcommand{\HF}{\mathtt{HF}}
\newcommand{\DW}{\mathtt{DW}}
\newcommand{\BP}{\mathtt{BP}}
\newcommand{\LL}{\mathcal{L}}
\newcommand{\DLC}{\mathtt{DLC}}
\newcommand{\REF}{\mathtt{REF}}

\makeatletter
    
    \@addtoreset{equation}{section}
  \makeatother

\title[Loewner chains with quasiconformal extensions]
	{Loewner chains with quasiconformal extensions: an approximation approach}
\author[I. Hotta]
	{Ikkei Hotta}
\address{Department of Applied Science, Yamaguchi University 2-16-1 Tokiwadai, Ube 755-8611, Japan}
\email{ihotta@yamaguchi-u.ac.jp}
\subjclass[2010]{Primary 30C62; Secondary 30C35, 34M30}
\keywords{evolution family; quasiconformal mapping; Loewner chain; Loewner differential equation}
\thanks{This work was supported by JSPS KAKENHI Grant Numbers 17K14205}
\date{\today}

\begin{document}

	%
	%

\begin{abstract}
A new approach in Loewner Theory proposed by Bracci, Contreras, D{\'{\i}}az-Madrigal and Gumenyuk provides a unified treatment of the radial and the chordal versions of the Loewner equations. 
In this framework, a generalized Loewner chain satisfies the differential equation
$$
\frac{\de f_{t}(z)}{\de t} = (z - \tau(t))(1-\closure{\tau(t)}z)p(z,t)\frac{\de f_{t}(z)}{\de z},
$$
where $\tau : [0,\infty) \to \closure{\D}$ is measurable and $p$ is called a Herglotz function.
In this paper, we will show that if there exists a $k \in [0,1)$ such that $p$ satisfies
$$
|p(z,t) - 1| \leq k |p(z,t) + 1|
$$
for all $z \in \D$ and almost all $t \in [0,\infty)$, then for all $t \in [0,\infty)$ $f_{t}$ has a $k$-quasiconformal extension to the whole Riemann sphere. 
The radial case ($\tau =0$) and the chordal case ($\tau=1$) have been proven by Becker [J. Reine Angew. Math. \textbf{255} (1972), 23--43] and Gumenyuk and the author [Math. Z. \textbf{285} (2017), no.3, 1063--1089].
In our theorem, no superfluous assumption is imposed on $\tau \in \closure{\D}$.
As a key foundation of the proof is an approximation method using a continuous dependence of evolution families and Loewner chains.
\end{abstract}

\maketitle

\vspace{-10pt}

\tableofcontents

	%
\section{Introduction}
	%

\subsection{Classical Loewner Theory}

Since the initial work of Loewner \cite{Loewner:1923}, the theory of the Loewner differential equations has been applied to prove several deep results in various fields of mathematics, including the complete proof of the famous Bieberbach conjecture due to de Branges \cite{deBranges:1985}. 
In 2000, Schramm \cite{Schramm:2000} introduced the celebrated Schramm-Loewner Evolution (SLE) which describes the scaling limits of various critical statistical mechanics models exhibiting conformal invariance. 
Recently, a new intrinsic approach to treat all the Loewner type equations was proposed by Bracci, Contreras, D{\'{\i}}az-Madrigal and Gumenyuk.

The \textit{(classical) radial case} is firstly introduced by L\"owner and later developed by Kufarev \cite{Kufarev:1943} and Pommerenke \cite{Pom:1965}.
Let $f_{t}(z) = e^{t}z + \sum_{n=2}^{\infty}a_{n}(t)z^{n} \,(t \geq 0)$ be a time-parameterized holomorphic function defined on the unit disk $\D := \{z \in \C : |z| < 1\}$ in the complex plane $\C$.
$(f_{t})_{t \ge 0}$ is said to be a \textit{(classical) radial Loewner chain} if for each $t \in [0,\infty)$ $f_{t}$ is univalent in $\D$ and the inclusion relation $f_{s}(\D) \subset f_{t}(\D)$ holds for all $0 \leq s < t < \infty$.
The key properties of radial Loewner chains are that $(f_t)$ is absolutely continuous on $t \in [0,\infty)$ for each $z \in \D$, which implies that $\de_t f_t \,\,(\de_t := \de / \de t) $ exists almost everywhere on $[0,\infty)$, and satisfies the partial differential equation
\begin{equation}\label{LKPDE}
\de_t f_t(z) = z\de_z f_t(z) \cdot p(z,t)
\end{equation}
for all $z \in \D$ and almost all $t \in [0, \infty)$, where $p$ is holomorphic on $z \in \D$ for each $t \in [0,\infty)$ and measurable on $t \in [0,\infty)$ for each $z \in \D$ satisfying $p(0,t) = 1$ and $\textup{Re\,} p(z,t)>0$ for all $z \in \D$ and $t \in [0,\infty)$.

\subsection{Quasiconformal extension}
In the radial case, any function $f$ belonging to $\SS$, the family of all univalent holomorphic functions $f$ on $\D$ with $f(0)=0$ and $f'(0)=1$, can be embedded in a certain radial Loewner chain $(f_t)$ as $f = f_0$ (\cite[Theorem 6.1]{Pom:1975}).
Thus, a wide spectrum of properties of the class $\SS$ can be derived via the radial Loewner equations. 
In particular, this paper focuses on quasiconformal extensions of univalent holomorphic functions.

A sense-preserving homeomorphism $f$ of a plane domain $G \subset \C$ is said to be \textit{$k$-quasiconformal} if $\de_{z} f$ and $\de_{\bar{z}} f$ in the distributional sense are locally integrable on $G$ and fulfill $|\de_{\bar{z}} f| \leq k |\de_{z} f|$ almost everywhere in $G$, where $k$ is a constant with $k \in [0,1)$.
For an introduction to the theory of quasiconformal mappings and related topics, see \cite{Ahlfors:2006}, \cite{LehtoVirtanen:1973} and \cite[Chapter 4]{ImayoshiTaniguchi:1992}.
For a given $f \in \SS$, if there exists a $k$-quasiconformal mapping $F$ of $\C$ such that $F = f$ on $\D$, then we say that \textit{$f$ has a $k$-quasiconformal extension to $\C$}.
Quasiconformal extendible univalent holomorphic functions were first treated by Bers (\cite{Bers:1961}) in connection with research on Teichm\"uller theory (see \cite{Lehto:1987}, \cite{Hubbard:2006}).
The first quasiconformal extension criterion was obtained by Ahlfors and Weill in 1962 (\cite{AhlforsWeill:1962}).
The reader is referred to \cite{Hotta:2009,Hotta:2010a} for the classical results on quasiconformal extensions of univalent functions.

In 1972, Becker \cite{Becker:1972} discovered a criterion for quasiconformal extension of $f \in \SS$ by means of the radial Loewner chain;
\textit{If $(f_{t})$ is a radial Loewner chain whose radial Herglotz function $p(z,t) := \de_{t}f_{t}(z)/z\de_{z}f_{t}(z)$ satisfies 
\begin{equation*}
\label{becker}
\left|
\frac{p(z,t)-1}{p(z,t)+1}
\right| \leq k <1 \hspace{10pt} (z \in \D, \mathrm{a.e.}~t \ge 0),
\end{equation*} 
then $f_{0}$ has a $k$-quasiconformal extension to $\C$} (see also \cite{Becker:1980}). The chordal variant of this theorem is obtained by Gumenyuk and the author \cite{HottaGumenyuk}.

In 1992, Betker \cite{Betker:1992} generalized Becker's theorem that, if there are two radial Herglotz functions $p$ and $q$ that satisfy
$$
\left|
\frac{p(z,t) - \closure{q(z,t)}}{p(z,t)+q(z,t)}
\right| \leq k<1 \hspace{10pt} (z \in \D, \mathrm{a.e.}~t \ge 0),
$$
then a Loewner chain associated with $p$ and an inverse Loewner chain (see \cite{Betker:1992} for detail) associated with $q$ defines a $k$-quasiconformal automorphisms of $\C$ whose restriction to $\D$ is $f_{0}$.

\subsection{Main results}

The main aim of this paper is to discuss quasiconformal extension problems on the general setting of Loewner Theory proposed in \cite{BracciCD:evolutionI} and \cite{MR2789373}.
In this framework, the following generalized Loewner chains are dealt with.
\begin{knowndefinition}[{\cite[Definition 1.2]{MR2789373}}]
		\label{defLoewner}
A family of holomorphic maps $(f_{t})_{t \geq 0}$ of the unit disk $\D$ is called a \textit{Loewner chain} if
\def\labelenumi{LC\arabic{enumi}.}
\begin{enumerate}
\item $f_{t} :\D\to\C$ is univalent for each $t \in [0,\infty)$;
\item $f_{s}(\D) \subset f_{t}(\D)$ for all $0 \leq s < t < \infty$;
\item for any compact set $K \subset \D$ and all $T>0$, there exists a non-negative locally integrable function $k_{K,T} : [0,T] \to \R$ such that
\begin{equation*}
|f_{s}(z) - f_{t}(z)| \leq \int_{s}^{t} k_{K,T}(\zeta) d\zeta
\end{equation*}
for all $z \in K$ and all $0 \leq s \leq t \leq T$.
\end{enumerate}
Furthermore, a Loewner chain is said to be \textit{normalized} if $f_{0} \in \SS$.
\end{knowndefinition}

It is known \cite{MR2789373} that a Loewner chain $(f_{t})$ in the above sense satisfies the differential equation
\begin{equation}
\label{opLDE}
\de_{t}f_{t}(z) = (z - \tau(t))(1-\closure{\tau(t)}z)\de_{z}f_{t}(z) \cdot p(z,t)\hspace{15pt}(z \in \D,\,\textrm{a.e.}~t \ge 0),
\end{equation}
where $\tau : [0,\infty) \to \closure{\D}$ is a measurable function and $p$ is called a \textit{Herglotz function} (Definition \ref{HFdef}).
Conversely, for a given $\tau$ and $p$, the ordinary differential equation
$$
\frac{\mathrm{d}\omega_{z,s} (t)}{\mathrm{d}t} = (\omega_{z,s}(t) - \tau(t))(\closure{\tau(t)}\omega_{z,s}(t) -1)\cdot p(\omega_{z,s}(t),t)\hspace{15pt}(\textrm{a.e.}~t \ge s)
$$
with the initial condition $\omega_{z,s} (s) =z$ has a unique solution $\omega_{z,s}(t)$.
Let $\varphi_{s,t}(z) :=w_{z,s}(t)$ for all $0 \le s \le t < \infty$ and all $z \in \D$, then the family of two-parametrized holomorphic functions $(\varphi_{s,t}(z))_{0 \le s \le t < \infty}$ on $z \in \D$ generates a Loewner chain $(f_{t})$  that fulfills $f_{s} = f_{t} \circ \varphi_{s,t}$ for all $0 \le s \le t < \infty$.
If we further assume that $(f_{t})$ is range-normalized (see Section \ref{2.3}), then such a chain $(f_{t})$ is determined uniquely.

In this paper, the following general quasiconformal extension criterion for Loewner chains is proven.
Some of the terminology below will be defined in more detail later on in the paper.
We emphasize that our theorem imposes no superfluous assumptions on $\tau$.

\begin{thm}\label{main01}
Let $k \in [0,1)$.
Let $(f_{t})$ be a Loewner chain and $(p, \tau)$ be the Berkson-Porta data associated with $(f_t)$.
We denote by $T \in [0,\infty]$ the smallest number such that $p(\D,t) \in i\R$ for almost all $t \in (T, \infty)$.
Suppose that 
\begin{enumerate}
\item $T > 0$;
\item There exists a compact subset of $\H :=\{w \in \C : \Re w > 0\}$ that contains $p(\D, t)$ for almost all $t \in [0,T)$;
\item $p$ satisfies
\begin{equation}\label{maininequality}
|p(z,t) - \closure{q(z,t)}|
\leq k \cdot |p(z,t) + q(z,t)|
\end{equation}
for all $z \in \D$ and almost all $t \in [0,T)$, where $q$ is a Herglotz function.
\end{enumerate}
Let $(g_{t})$ be a decreasing Loewner chain associated uniquely with $(q, \tau)$.
Then for each $t \in [0, T)$, $f_{t}$ and $g_{t}$ has continuous extensions to $\closure{\D}$.
Further, $\Phi$ defined by 
\begin{equation}\label{Phi}
\left\{
\begin{array}{lll}
\dstyle \Phi(z) := f_{0}(z), & z \in \D,\\
\dstyle \Phi\left(\frac{1}{\closure{g_{t}(e^{i\theta})}}\right) := f_{t}(e^{i\theta}), &\theta \in [0,2\pi)\hspace{5pt}\textup{and}\hspace{5pt}t \in [0,T),
\end{array}
\right.
\end{equation}
is a $k$-quasiconformal mapping on $\left\{1/\bar{w}: w \in \CC\backslash\bigcap_{t \geq 0} \closure{g_t(\D)}\right\}$ onto $\bigcup_{t \ge 0}f_{t}(\D)$.
\end{thm}

\begin{rem}
\label{rem02}
The inequality \eqref{maininequality} implies that for a fixed $(z_{0}, t_{0})  \in \D\times [0,T)$, $p(z_{0}, t_{0})$ lies on a circle of Apollonius on $\H$ with foci $\closure{q(z_{0},t_{0})}$ and $-q(z_{0}, t_{0})$, symmetric w.r.t. the imaginary axis.
Hence, $q(\D,t)$ is also bounded for almost all $t \in [0,T)$.
\end{rem}

As a corollary of the above theorem, we obtain a generalization of Becker's criterion. Note that recently another simple proof of the next theorem relying on the technique of holomorphic motions and the optimal lambda-lemma is provided by Gumenyuk and Prause \cite{GumenyukPrause:2018}.

\begin{thm}
\label{openingthm}
Let $k \in [0, 1)$.
Let $(f_t)$ be a Loewner chain and $p$ be the Herglotz function associated with $(f_t)$ in \eqref{opLDE}.
Suppose that $p$ satisfies
\begin{equation*}
\label{openingineq}
\left|
\frac{p(z,t)-1}{p(z,t)+1}
\right| \leq k
\end{equation*}
for all $z \in \D$ and almost all $t \in [0,\infty)$. Then:
\begin{enumerate}
\item[(i)] for each $t \in [0,\infty)$, $f_{t}$ has a $k$-quasiconformal extension to $\CC$;
\item[(ii)] for each $s \in [0,\infty)$ and $t \in [s,\infty)$, $\varphi_{s,t}:= f_{t}^{-1} \circ f_{s}$ has a $k$-quasiconformal extension to $\CC$;
\item[(iii)] $\bigcup_{t \ge 0}f_{t}(\D) = \C$.
\end{enumerate}
\end{thm}

For proving Theorem \ref{main01}, we introduce an approximation method for Loewner chains.
A main part of the statement of the following theorem and the proof (Lemma \ref{HVweaktoEF} in Section 3) come from \cite[Lemma I.37]{Roth:1998}. 
Similar results are obtained under more restrictive situations (see e.g., \cite[Section 4.7]{Lawler:2005}, \cite[Proposition 1]{Viklund2012} and \cite[Theorem 2.4]{RothSchleissinger:2014}).

\begin{thm}
Let $\tau_{n} : [0,\infty) \to \closure{\D}$ be a sequence of measurable functions and $p_{n}$ a sequence of Herglotz functions.
Suppose that $G_{n}(z,t) := (z - \tau_{n}(t))(\closure{\tau_{n}(t)}z-1)  p_{n}(z,t)$ has the following properties;
\begin{enumerate}
\item For all $z \in \D$, $G_{n}$ has a weak limit $G$ of the form $G(z,t) := (z - \tau(t))(\closure{\tau(t)}z-1) p(z,t),$ where $\tau: [0,\infty) \to \closure{\D}$ and $p$ are again a measurable function and a Herglotz function; 
\item For all $T >0$, $\{G_{n}(\,\cdot\,, t) : n \in \N, \mathrm{a.e.}~t \in [0,T]\}$ forms a normal family.
\end{enumerate}
Let $\omega_{z,s}$ be the unique solution of the ordinary differential equation
$$
\left\{
\begin{array}{clll}
\dstyle\frac{\mathrm{d}\omega_{z,s} (t)}{\mathrm{d}t} &=& G(\omega_{z,s}(t), t), &\textit{for almost all } t \ge s\\[8pt]
\omega_{z,s}(s) &=& z, & t=s
\end{array}
\right.
$$
and $\omega_{z,s}^{n}$ the unique solution of the above ODE by $G_{n}$ in the same fashion. 
Let $\varphi_{s,t}(z) := \omega_{z,s}(t)$ and $\varphi_{s,t}^{n}(z) := \omega_{z,s}^{n}(t)$ for all $0 \le s \le t < \infty$ and all $z \in \D$.
Then $(\varphi_{s,t}^{n})_{0 \le s \le t < \infty}$ converges to $(\varphi_{s,t})_{0 \le s \le t < \infty}$ locally uniformly on $(z,t) \in \D \times [s,\infty)$ as $n \to \infty$.

Further, let $(f_{t}^{n})$ and $(f_{t})$ be range-normalized Loewner chains associated uniquely with $(\varphi_{s,t}^{n})$ and $(\varphi_{s,t})$, respectively.
Then $(f_{t}^{n})$ converges locally uniformly to $(f_{t})$ on $(z,t) \in \D \times [0,\infty)$ as $n \to \infty$.
\end{thm}

This paper is structured as follows:
In Section \ref{Section2}, we collect preliminary results that are used throughout the later discussions. 
In Section \ref{Section3}, we prove the quasiconformal extension theorem (Theorem \ref{main01}) with Loewner chains and decreasing Loewner chains (Definition \ref{def_decreasing}). 
The proof is divided into three steps where $\tau$ is a constant, a step function, and a measurable function on $\closure{\D}$, respectively.
An approximation method for Loewner chains is also provided in this section.
In Section \ref{Section4}, we verify Theorem \ref{openingthm} and further results of quasiconformal extensions that are corollaries of the theorem in Section 3. We conclude Section \ref{Section4} and this paper with a brief consideration of the Loewner Range $\bigcup_{t \ge 0} f_{t}(\D)$.

\

\no
\textbf{Acknowledgement}: The author would like to express his deepest gratitude to Professor Oliver Roth and Doctor Sebastian Schlei{\ss}inger for their fruitful discussions and suggestions on this paper.
The author is deeply grateful to Professor Pavel Gumenyuk who gave him continuous encouragement and valuable advice.
He would also like to thank the anonymous referee for thorough reading of the manuscript and helpful comments.
Part of this work was done while the author was a guest researcher at the University of W\"urzburg.

	%
\section{Preliminaries}\label{Section2}

\subsection{Semigroups of holomorphic mappings}

Let $D \subset \C$ be a simply connected domain.
We denote the family of all holomorphic functions on $D$ by $\hol(D,\C)$.
If $f \in \hol(\D, \C)$ is a self-mapping of $\D$, then we denote the family of such functions by $\hol(\D)$.

A family $\{\phi_t\}_{t \geq 0}$ of holomorphic self-maps of $\D$ is called a \textit{one-parameter \textup{(}continuous\textup{)} semigroup} if 
\begin{itemize}
\item $\phi_0 = id_{\D}$;
\item $\phi_{s+t} = \phi_t \circ \phi_s$ for all $s, t \in [0, \infty)$;
\item $\lim_{t \to s} \phi_t (z) = \phi_s (z)$ for all $s \in [0,\infty)$ and $z \in \D$;
\item $\lim_{t \to 0^+} \phi_t (z) = z$ locally uniformly on $\D$.
\end{itemize}

It is well-known (see e.g. \cite{ElinShoikhet:2010}) that for a semigroup $\phi_t$ there exists a holomorphic function $G \in \hol(\D, \C)$ such that $\phi_t$ is the unique solution of the ordinary differential equation
\begin{equation*}\label{cauchy_semigroup}
\dstyle \frac{{\rm d}\phi_t(z)}{{\rm d}t} = G(\phi_t(z))\hspace{20pt}(t \ge 0)
\end{equation*}
with the initial condition $\phi_0(z) = z$.
The above function $G$ is called an \textit{infinitesimal generator} of a semigroup.
Various criteria which guarantee that a function $G \in \hol(\D, \C)$ is an infinitesimal generator are known.
As one of them, in 1978 Berkson and Porta \cite{BerksonPorta:1978} showed that a holomorphic function $G \in \hol(\D, \C)$ is an infinitesimal generator if and only if there exists a $\tau \in \closure{\D}$ and a function $p \in \hol(\D, \C)$ with $\Re p(z) \geq 0$ for all $z \in \D$ such that
\begin{equation}\label{BPformula}
G(z) = (\tau-z)(1-\bar{\tau} z)p(z)
\end{equation}
for all $z \in \D$. 
\eqref{BPformula} is called the \textit{Berkson-Porta representation}.

\subsection{Evolution families and Herglotz vector fields}

We introduce an evolution family, a core notion in modern Loewner Theory.

\begin{knowndefinition}[{\cite[Definition 3.1]{BracciCD:evolutionI}}]
A family of holomorphic self-maps of the unit disk $(\varphi_{s,t}),\,0 \leq s \leq t < \infty$, is an \textit{evolution family} if
\begin{enumerate}
\def\labelenumi{EF\arabic{enumi}.}
\item $\varphi_{s,s}(z) = z$;
\item $\varphi_{s,t} = \varphi_{u,t}\circ \varphi_{s,u}$ for all $0 \leq s \leq u \leq t < \infty$;
\item for all $z \in \D$ and for all $T>0$ there exists a
locally integrable function ${k_{z,T}\colon[0,+\infty)\to[0,+\infty)}$ such that
\[
|\varphi_{s,u}(z)-\varphi_{s,t}(z)|\leq\int_{u}^{t}k_{z,T}(\xi)d\xi
\]
whenever $0\leq s\leq u\leq t<\infty.$
\end{enumerate}
\end{knowndefinition}

\begin{rem}
In~\cite{BracciCD:evolutionI} and \cite{MR2789373}, an evolution family and a Loewner chain are defined with an integrability order $d\in[1,+\infty]$. 
Since this parameter is not important for the discussions in this paper, we assume that $d=1$ which is the most general case of the order.
\end{rem}

We denote the family of all evolution families by $\EF$.
For all $0 \leq s \leq t < \infty$, $\varphi_{s,t}$ is univalent on $\D$ {\cite[Corollary 6.3]{BracciCD:evolutionI}.

\begin{knowndefinition}[{\cite[Definition 4.1, Definition 4.3]{BracciCD:evolutionI}}]
\label{WVdef} 
A function $G : \D \times [0,\infty) \to \C$ is said to be a \textit{Herglotz vector field} and denoted by $G \in \HV$ if it satisfies the following three conditions:
\def\labelenumi{HV\arabic{enumi}.}
\begin{enumerate}
\item for all $z \in \D$, the function $G(z,\,\cdot\,)$ is measurable on $t \in [0,\infty)$;
\item for any compact set $K \subset \D$ and for all $T>0$, there exists a non-negative locally integrable function $k_{K,T} : [0,T] \to\R$ such that
\begin{equation*}
\label{WVinequality}
|G(z,t)| \leq k_{K,T}(t)
\end{equation*}
for all $z \in K$ and for almost every $t \in [0,T]$;
\item $G(\,\cdot\,, t)$ is an infinitesimal generator of a semigroup of holomorphic functions for almost all $t \in [0,\infty)$.
\end{enumerate}
\end{knowndefinition}

There is one-to-one essentially unique correspondence between evolution families and Herglotz vector fields; 
for any $(\varphi_{s,t}) \in \EF$, there exists an essentially unique $G \in \HV$ such that
\begin{equation}
\label{EFtoHV}
\frac{{\rm d}\varphi_{s,t}(z)}{{\rm d}t} = G(\varphi_{s,t}(z), t)
\end{equation}
for all $z \in \D$ and almost all $t \in [0,\infty)$.
Conversely, for any $G \in \HV$, the family of unique solutions of \eqref{EFtoHV} with the initial condition $\varphi_{s,s}(z) =z$ is an evolution family.
Here, \textit{essentially unique} means that if $G^{*}(z,t)$ is another Herglotz vector field which satisfies \eqref{EFtoHV}, then $G(\,\cdot\,,t) = G^{*}(\,\cdot\,,t)$ for almost every $t \ge 0$.

\begin{knowndefinition}[{\cite[Definition 4.5]{BracciCD:evolutionI}}]
			\label{HFdef}
\textit{A Herglotz function} on the unit disk $\D$ is a function $p : \D \times [0,\infty) \to \C$ with the following properties:
\def\labelenumi{HF\arabic{enumi}.}
\begin{enumerate}
\item for all fixed $z \in \D$, the function $p(\,\cdot\,,t)$ is locally integrable on $[0,\infty)$ for all $z\in\D$;
\item for all fixed $t \in [0,\infty)$, the function $p(z,\,\cdot\,)$ is holomorphic on $\D$;
\item $\Re p(z,t) \geq 0$ for all $z \in \D$ and $t \in [0,\infty)$. 
\end{enumerate}
\end{knowndefinition}

\no
$\HF$ stands for the family of all Herglotz functions.

A mutual relation also holds between Herglotz vector fields and the Berkson-Porta datas; 
for any $G \in \HV$, there exist a measurable function $\tau : [0,\infty) \to \closure{\D}$ and $p \in \HF$ such that
\begin{equation}\label{HVtoBP}
G(z,t) = (z-\tau(t)) (\closure{\tau(t)} z -1) p(z,t)
\end{equation}
for all $z \in \D$ and almost all $t \in [0,\infty)$.
Conversely, for a given measurable function $\tau : [0,\infty) \to \closure{\D}$ and $p \in \HF$, equation \eqref{HVtoBP} forms a Herglotz vector field. 
For our convenience, in this paper we call the above measurable function $\tau : [0,\infty) \to \closure{\D}$ the \textit{Denjoy-Wolff function}, and denote by $ \tau \in \DW$.
A pair $(p, \tau)$ of $p \in \HF$ and $\tau \in \DW$ is called the \textit{Berkson-Porta data} for $(\varphi_{s,t})$, and denoted by $(p, \tau) \in \BP$.
If two $(p,\tau) ,(\tilde{p},\tilde{\tau}) \in \BP$ generate the same $G \in \HV$ up to a set of measure zero, then $p = \tilde{p}$ for all $z \in \D$ and almost all $t \in [0,\infty)$ and $\tau = \tilde{\tau}$ for almost all $[t,\infty)$ such that $G(\,\cdot\,,t) \not\equiv 0$ (see \cite[Theorem 4.8]{BracciCD:evolutionI}).

In summary, there is a one-to-one correspondence among an evolution family $(\varphi_{s,t}) \in \EF$, a Herglotz vector fields $G \in \HV$ and the Berkson-Porta data $(p,\tau) \in \BP$. 
In particular, the relation of $(\varphi_{s,t})$ and $(p,\tau)$ are expressed by the ordinary differential equation
\begin{equation}\label{evolution}
\frac{{\rm d}\varphi_{s,t}(z)}{{\rm d}t}= (\varphi_{s,t}(z) - \tau(t))(\closure{\tau(t)}\varphi_{s,t}(z)-1) p(\varphi_{s,t}(z),t)\hspace{15pt}(z \in \D,\,\textrm{a.e.}~t \ge s)
\end{equation}
with the initial condition $\varphi_{s,s}(z) = z$.

\subsection{Generalized Loewner chains and decreasing Loewner chains}

\label{2.3}

By Definition \ref{defLoewner}, the notion of Loewner chains is lifted to the same framework as evolution families.
The family of all Loewner chains is denoted by $\LC$.

For a given $(f_{t}) \in \LC$, the equation $\varphi_{s,t} := f_{t}^{-1} \circ f_{s}$ defines an evolution family.
Differentiating both sides with respect to $t$, we obtain $\de_{z}f_{t} \cdot \de_{t}\varphi_{s,t} + \de_{t} f_{t} = 0$. 
Taking into account of \eqref{evolution} we have the following generalized Loewner-Kufarev PDE
\begin{equation}
\label{LdLDE}
\de_{t}f_{t}(z) = (z-\tau(t))(1-\closure{\tau(t)}z)\de_{z}f_{t}(z)p(z,t).
\end{equation}

On the other hand, for a given evolution family, the relation $f_{t} \circ \varphi_{s,t} = f_{s}$ does not define a unique Loewner chain.
That is, it is not always true that $\LL[(\varphi_{s,t})]$, the family of all normalized\footnote{Recall that $(f_{t}) \in \LC$ is \textit{normalized} if $f_{0} \in \SS$ (Definition \ref{defLoewner}).} Loewner chains associated with $(\varphi_{s,t}) \in \EF$, consists of one function.
However, $\LL[(\varphi_{s,t})]$ always includes a certain Loewner chain and in this sense $(f_t)$ is determined uniquely.
For the detail, see {\cite[Theorem 1.6, Theorem 1.7]{MR2789373}}.
The uniquely determined Loewner chain $(f_{t})$ in such a way is called \textit{range-normalized} and denoted by $(f_{t}) \in \LC_{0}$.
It should be noted here that the \textit{Loewner range} defined by 
\begin{equation*}
\label{loewner-range-def}
\Omega[(f_{t})] := \bigcup_{t \ge 0}f_{t}(\D).
\end{equation*}
is concerned to the uniqueness of $(f_{t})$.

In Section \ref{Section3}, we will use the decreasing setting of evolution families and Loewner chains.

\begin{knowndefinition}[{\cite[Definition 1.6]{contreraslocalduality}}]
\label{def_decreasing}
A family $(g_{t})_{t \geq 0}$ of holomorphic maps of the unit disk $\D$ is called a \textit{decreasing Loewner chain} if it satisfies the following conditions:
\def\labelenumi{DLC\arabic{enumi}.}
\begin{enumerate}
\item $g_{t} $ is univalent on $\D$ for each $t \in [0, \infty)$;
\item $g_{0}(z) = z$ and $g_{s}(\D) \supset g_{t}(\D)$ for all $0 \leq s < t < \infty$;
\item for any compact set $K \subset \D$ and all $T>0$, there exists a 
locally integrable function ${k_{K,T}\colon[0,+\infty)\to[0,+\infty)}$ such that
\begin{equation*}
\label{Lddecreasingineq}
|g_{s}(z) - g_{t}(z)| \leq \int_{s}^{t} k_{K,T}(\zeta) d\zeta
\end{equation*}
for all $z \in K$ and all $0 \leq s \leq t \leq T$.
\end{enumerate}
\end{knowndefinition}

Accordingly, the decreasing counterpart of evolution families is also defined.

\begin{knowndefinition}[{\cite[Definition 1.9]{contreraslocalduality}}]
A family $(\omega_{s,t})_{0 \le s \le t}$ of holomorphic self-maps of the unit disk $\D$ is called a \textit{reverse evolution family} if the following conditions are fulfilled:
\def\labelenumi{REF\arabic{enumi}.}
\begin{enumerate}
\item $\omega_{s,s}(z) = z$;
\item $\omega_{s,t} = \omega_{s,u} \circ \omega_{u,t}$ for all $0 \le s \le u \le t < \infty$;
\item for all $z \in \D$ and for all $T>0$ there exists
 a non-negative locally integrable function ${k_{z,T}\colon[0,+\infty)\to[0,+\infty)}$ such that
$$
|\omega_{s,u}(z) - \omega_{s,t}(z)| \leq \int_{u}^{t} k_{z, T}(\zeta) d\zeta
$$
for all $0 \leq s \leq u \leq t \leq T$.
\end{enumerate}
\end{knowndefinition}

$\DLC$ and $\REF$ denote the families of all decreasing Loewner chains and reverse evolution families, respectively.

The followings state the mutual relations between decreasing Loewner chains, reverse evolution families and Herglotz vector fields:
\begin{equation}
\label{uniqueDLE}
\omega_{s,t} = g_s^{-1} \circ g_t\hspace{15pt} (0 \leq s \leq t < \infty)
\end{equation}
and
\begin{equation}
\label{reversetoHVF}
\frac{{\rm d}\omega_{s,t}(z)}{{\rm d}s} = -G(\omega_{s,t}(z),s),\,\, s \in [0,t],\,\,\omega_{s,s}(z) = z.
\end{equation}
For the detail, see \cite[Theorem 4.1, Theorem 4.2]{contreraslocalduality}.
Thus, $(\omega_{s_t}) \in \REF$ and $(g_t) \in \DLC$ satisfy the following differential equations
\begin{equation*}
\label{LdrevODE}
\frac{{\rm d}\omega_{s,t}(z)}{{\rm d}t} = (\omega_{s,t}(z) - \sigma(t))(1-\closure{\sigma(t)}\omega_{s,t}(z))q(\omega_{s,t}(z),t)\hspace{15pt}(z \in \D,\,\textrm{a.e.}~t \ge s),
\end{equation*}
and
\begin{equation}
\label{LddecLDE}
\de_{t}g_{t}(z) = (z-\sigma(t))(\closure{\sigma(t)}z -1)\de_{z}g_{t}(z)q(z,t)\hspace{15pt}(z \in \D,\,\textrm{a.e.}~t \ge 0),
\end{equation}
where $q \in \HF$ and $\sigma \in \DW$.

For $(g_t) \in \DLC$, an intersection of all image of $\D$ under $(g_t)$ for $t \geq 0$ is denoted by $\Lambda[(g_t)]$, i.e.,
$$
\Lambda[(g_t)] := \bigcap_{t \geq 0} \closure{g_t(\D)}.
$$
In the work by Betker \cite{Betker:1992}, an image of an inverse Loewner chain is assumed to shrink to the origin as $t \to \infty$.
On the other hand, the situation is rather complicated in the case of decreasing Loewner chains.
One cannot even expect that $\Lambda[(g_{t})]$ is a simply-connected domain.

Lastly, we define $\Delta[(g_{t})]$ by
$$
\Delta[(g_{t})] := \left\{\frac1{\bar{w}}: w \in \CC\backslash\Lambda[(g_t)]\right\}.
$$

	%
\section{General quasiconformal extension criterion}\label{Section3}

In this section we give a proof of Theorem \ref{main01}.

\subsection{Approximation lemmas for $(\varphi_{s,t})$ and $(f_{t})$}

In order to show Theorem \ref{main01}, we need the following approximation lemmas for evolution families and Loewner chains.

Our first lemma is a special case of \cite[Lemma I.37]{Roth:1998}.
Since in our setting some arguments are simplified, we reconstruct the proof of the lemma\footnote{The author wishes to thank Professor Oliver Roth for his suggestion to write a proof of the lemma.}.

\begin{lem}
\label{HVweaktoEF}
Suppose that $\Gamma := \{G_{n}(\,\cdot\,, t) : \mathrm{a.e.}~t \in [0,T], n \in \N\} \subset \HV$ forms a normal family for all $T>0$.
If $\{G_{n}\}_{n \in \N} \subset \Gamma$ is a sequence converging weakly to $G \in \HV$, then a sequence of evolution families $(\varphi_{s,t}^{n})$ associated with $G_{n}$ converges locally uniformly to $(\varphi_{s,t})$ associated with $G$ on $(z,t) \in \D \times [s,\infty)$.
\end{lem}

\no
Here, a sequence of Herglotz vector fields $\{G_{n}\}_{n} \in \HV$ as holomorphic functions on $\D$ converges \textit{weakly} to $G \in \HV$ on $I$ means that for any compact subset $I' \subset I$,
$$
\int_{I'} G_{n}(z,u) du \to \int_{I'}G(z,u)du
$$
as $n \to \infty$.
We remark that the sequence $\{G_{n}\} \subset \HV$ always has a weakly convergent subsequence with weak limit $G \in \HV$ (see \cite[Corollary I.30]{Roth:1998}).

In the proof, we will use the following Gronwall type inequality;

\begin{knownlem}[{See e.g. \cite{Dragomir:2003}}]
\label{gronwall}
Let $\theta \in L^{\infty}([a,b]; \R),\,k \in L^{1}([a,b];\R)$ non-negative and $f:[a,b] \to \R$ increasing. If
$$
\theta(t) \leq f(t) + \int_{a}^{t}k(\xi)\theta(\xi) d\xi, \hspace{15pt} t \in [a,b],
$$
then
$$
\theta(t) \leq f(t) \exp\left(\int_{a}^{t}k(\xi) d\xi\right),\hspace{15pt} t \in [a,b].
$$
\end{knownlem}

\begin{proof}[\textbf{Proof of Lemma \ref{HVweaktoEF}}]
Let $K$ be a compact subset in $\D$ and $T > s$. 
We prepare some notations.
Let $O_{K,T} := \{\varphi_{s,t}(z) : z \in K,\, t \in [s,T] \}$.
Since $O_{K,T}$ is compact, there exist real constants $0 < b < b' < \infty$ such that
$$
O_{K,T} \subsetneq A_{K,T} \subsetneq B_{K,T} \subsetneq \D,
$$
where
$$
A_{K,T} := \{ w \in \D : d(w, O_{K,T}) < b\}
$$
and
$$
B_{K,T} := \{ w \in \D : d(w, O_{K,T}) \le b'\}.
$$
Then, take $\vareps \in (0,b)$.

\no
\textbf{(a).} 
Let
$$
\alpha_{n}(z,t) := \int_{s}^{t}\left[G(\varphi_{s,u}(z), u) - G_{n}(\varphi_{s,u}(z), u) \right] du.
$$
Since $G_{n}$ converges weakly to $G$, $\alpha_{n}(z,t)$ tends to 0 pointwise on $(z,t) \in K \times [s,T]$ as $n \to \infty$.

We will show that the above convergence is uniform on $K \times [s,T]$.
By the normality of $\Gamma$, one can find a uniform constant $M_{K,T}$ depending only on $K$ and $T$ such that
$$
|G_{n}(w_{1}, t) - G_{n}(w_{2}, t)| \le M_{K,T}|w_{1}-w_{2}|
$$
for all $w_{1}, w_{2} \in B_{K,T}$, almost all $t \in [s,T]$ and all $n \in \N$.
Moreover, there exists a constant $M_{K,T}'$ such that $|G_{n}(w,t)| \le M_{K,T}'$ for all $w \in B_{K,T}$, almost all $t \in [s,T]$ and $n \in \N$.
Therefore,
\begin{eqnarray*}
\lefteqn{\left|\int_{s}^{t_{1}}G_{n}(w_{1}, u) du - \int_{s}^{t_{2}}G_{n}(w_{2}, u) du\right|}\\
	&\leq&
\left|\int_{s}^{t_{1}}\left[G_{n}(w_{1}, u) - G_{n}(w_{2}, u)\right] du\right| 
+ \left|\int_{t_{1}}^{t_{2}}G_{n}(w_{2}, u) du\right|\\[5pt]
	&\leq&
M_{K,T}|w_{1}-w_{2}|\cdot|t_{1}-s| + M_{K,T}'|t_{1}-t_{2}| 
\end{eqnarray*}
for all $w_{1}, w_{2} \in B_{K,T}$, almost all $t_{1}, t_{2} \in [s,T]$ and all $n \in \N$.
It implies that the family $\{\int_{s}^{t}G_{n}(\varphi_{s,t}(z),u)du\}_{n \in \N}$ is equicontinuous on $K \times [s,T]$.
Hence $\alpha_{n}(z,t)$ converges to 0 uniformly on $K \times [s,T]$ as $n \to \infty$.

\no
\textbf{(b).} 
By Definition \ref{WVdef} and {\cite[Lemma 4.2]{BracciCD:evolutionI}}, one can choose a constant $M_{K,T}''$ such that
\begin{equation}
\label{GGninequality}
\begin{array}{lll}
|G(\varphi_{s,t}, t)| \leq M''_{K,T}, \hspace{15pt}|G_{n}(\varphi_{s,t}, t)| \leq M''_{K,T} \hspace{15pt} \textup{and}\\[6pt]
|G(\varphi_{s,t}, t) - G(\varphi_{s,t}', t)| \leq M''_{K,T}|\varphi_{s,t}-\varphi_{s,t}'|
\end{array}
\end{equation}
for all $w, w' \in B_{K,T}$, almost all $t \in [s,T]$ and all $n \in \N$.

Let $u > 0$ be a constant satisfying 
\begin{equation}
\label{Mbineq}
(\vareps + 2M_{K,T}''u)\exp(M_{K,T}''u) < b.
\end{equation}
We will prove that under inequality \eqref{Mbineq}, $w_{n}(z,t)(=\varphi_{s,t}^{n}(z))$ belongs to $B_{K,T}$ for all $z \in K$ and all $t \in [s,s + u]$.
Let us fix $n \in \N$ and let $u_{n}$ be the largest number such that $w_{n}(K, t)$ lies on $B_{K,T}$ for all $ t \in [s,s + u_{n}]$.
By assumption, $u_{n}$ is at least strictly greater than 0.
For all $t \in [s,s + u_{n}]$ we have
\begin{eqnarray*}
\lefteqn{|w(z,t)- w_{n}(z,t)| 
	=
\left|\int_{s}^{t}G(w(z,u), u) du - \int_{s}^{t}G_{n}(w_{n}(z,u), u)du\right|}\\
	&\le&
\int_{s}^{t}|G(w(z,u), u) du - G(w_{n}(z,u), u)|du
	+
\int_{s}^{t}|G(w_{n}(z,u), u) - G_{n}(w_{n}(z,u), u)|du\\
	&\le&
2M_{K,T}''u_{n} + M_{K,T}''\int_{s}^{t}|w(z,u)-w_{n}(z,u)| du.
\end{eqnarray*}
Applying Lemma \ref{gronwall}, $|w(z,t)- w_{n}(z,t)| < 2M_{K,T}''u_{n}\exp(M_{K,T}''u_{n})$ for all $z \in K$ and all $t \in [s,s + u_{n}]$.
Now, suppose that $2M_{K,T}''u_{n}\exp(M_{K,T}''u_{n}) < b$. 
Then $w_{n}(z,t) \in A_{K,T}$ for all $z \in K$ and all $t \in [s,s + u_{n}]$.
It, however, contradicts that $u_{n}$ is the largest number such that $w_{n}(K, t)$ lies on $B_{K,T}$ for all $ t \in [s,s + u_{n}]$ (note that $A_{K,T}$ is an open set and $B_{K,T}$ is compact).
Thus $b \le 2M_{K,T}''u_{n}\exp(M_{K,T}''u_{n})$.
Since $(\vareps + 2M_{K,T}''u)\exp(M_{K,T}''u) < b$, we obtain $u \leq u_{n}$.
It concludes that $w_{n}(z,t) \in B_{K,T}$ for all $z \in K$, all $t \in [s, s + u]$ and all $n \in \N$.

The above argument yields
$$
\left|\int_{s}^{t}G_{n}(\varphi_{s,u}(z), u) du - \int_{s}^{t}G_{n}(\varphi_{s,u}^{n}(z), u)du\right| \le M_{K,T}''\int_{s}^{t}|\varphi_{s,t}(z)-\varphi_{s,t}^{n}(z)|ds
$$
for all $z \in K$, all $t \in [s,s + u]$ and all $n \in \N$.

\no
\textbf{(c).}
By (a) and (b), we have
\begin{eqnarray*}
\lefteqn{|\varphi_{s,t}(z) - \varphi_{s,t}^{n}(z)|}\\
	&=&
\left|\int_{s}^{t}G(\varphi_{s,u}(z),u)du - \int_{s}^{t}G_{n}(\varphi_{s,u}^{n}(z), u)du\right|\\
	&\le&
\int_{s}^{t}|G(\varphi_{s,u}(z),u)du - G_{n}(\varphi_{s,u}(z), u)|du + \int_{s}^{t}|G_{n}(\varphi_{s,u}(z),u)du - G_{n}(\varphi_{s,u}^{n}(z), u)|du\\
	&\le&
\alpha_{n}(z,t) + M_{K,T}''\int_{s}^{t}|\varphi_{s,u}(z) - \varphi_{s,u}^{n}(z)| du
\end{eqnarray*}
for all $z \in K$, all $t \in [s,s + u]$ and all $n \in \N$.
Applying Lemma \ref{gronwall}, $\varphi_{s,t}^{n}(z)$ converges to $\varphi_{s,t}(z)$ uniformly on $(z,t) \in K \times [s,s + u]$.

\no
\textbf{(d).}
We finally prove that $\varphi_{s,t}^{n}(z)$ converges to $\varphi_{s,t}(z)$ uniformly on $(z,t) \in K \times [s,T]$.
Let us choose $N \in \N$ such that $|\varphi_{s,t}(z) - \varphi_{s,t}^{n}(z)| < \vareps$ for all $z \in K$, all $t \in [s,s + u']$ and all $n \in \N$, where $u'>0$.
Then we fix $n >N$ and define $u_{n}'$ as the largest number such that $\varphi_{s,t}^{n}(K)$ lies on $B_{K,T}$ for all $t \in [s,s + u_{n}']$, as in part (b).
Remark that $u_{n}' \ge u'$.
Then following the similar line as (b) with Lemma \ref{gronwall} we have
$$
|\varphi_{s,t}(z)-\varphi_{s,t}^{n}(z)| \le \left(|\varphi_{s,s + u'}(z) - \varphi_{s,s + u'}^{n}(z)| + 2(u_{n}'-u')M_{K,T}''\right)\exp(M_{K,T}''(u_{n}'-u')).
$$ 
for all $z \in K$ and all $t \in [s + u', s + u_{n}']$.
If we suppose $(\vareps + 2(u_{n}' -u')M_{K,T}'')\exp(M_{K,T}''(u_{n}'-u_{n})) < b$, then $\varphi_{s,t}^{n}(z) \in A_{K,T}$ for all $z \in K$ and all $t \in [s + u', s + u_{n}']$. 
So, we have $b \le (\vareps + 2(u_{n}' -u')M_{K,T}'')\exp(M_{K,T}''(u_{n}'-u'))$ and hence $u \le u_{n}' - u'$.
It concludes with the last inequality in \eqref{GGninequality} that
$$
\left|\int_{s}^{t} [G(\varphi_{s,u}(z),u) - G(\varphi_{s,u}^{n}(z), u)] du\right|
	\le
M_{K,T}''\int_{s}^{t}|\varphi_{s,u}(z)-\varphi_{s,u}^{n}(z)| du
$$
for all $z \in K$, all $t \in [s, s + u + u']$ and all $n >N$.
Repeating the argument in part (c)  we have
$$
|\varphi_{s,t}(z) - \varphi_{s,t}^{n}(z)|
	\le
\alpha_{n}(z,t) + M_{K,T}''\int_{s}^{t}|\varphi_{s,t}(z) - \varphi_{s,t}^{n}(z)| du
$$
for all $z \in K$, all $t \in [s,s + u + u']$ and all $n >N$.
Therefore $(\varphi_{s,t}^{n})$ converges to $(\varphi_{s,t})$ uniformly on $(z,t) \in K \times [s,s + u + u']$.
Since the interval $[s,T]$ is compact, we have prove that $(\varphi_{s,t}^{n})$ converges to $(\varphi_{s,t})$ uniformly on $(z,t) \in K \times [s, T]$.
\end{proof}

\begin{rem}
The opposite direction of Lemma \ref{HVweaktoEF} does not hold in general.
Let $(\varphi_{s,t}^{n})$ be a locally uniformly convergent sequence of evolution families and $(f_{t}^{n}) \in \LL[(\varphi_{s,t}^{n})]$. 
Then due to following Lemma \ref{seqEVtoLC}, for a fixed $t \in [0,\infty)$ $\{f_{t}^{n}\}_{n}$ (hence $\{\de_{z}f_{t}^{n}\}_{n}$ also) converges locally uniformly on $\D$. 
Now observing $G_{n}(z,t) := -\de_{z}f_{t}^{n}(z) / \de_{t}f_{t}^{n}(z)$, one cannot expect a nice convergent property to $\{\de_{t}f_{t}^{n}\}_{n}$, and therefore to $\{G_{n}\}_{n}$. 
See also \cite[Section 4.2]{LindMR:2009} in which a counterexample is presented in the chordal setting.
\end{rem}

\begin{lem}
\label{seqEVtoLC}
Let $(\varphi_{s,t}^{n})$ be a sequence of evolution families converging to $(\varphi_{s,t}) \in \EF$ locally uniformly on $(z,t) \in \D \times [0,\infty)$.
Let $(f_{t}^{n}) \in \LC_{0}$ and $(f_{t}) \in \LC_{0}$ associated with $(\varphi_{s,t}^{n})$ and $(\varphi_{s,t})$, respectively.
Then, $(f_{t}^{n})$ converges locally uniformly to $(f_{t})$ on $(z,t) \in \D \times [0,\infty)$.
\end{lem}

\begin{proof}
Following to \cite[Theorem 3.3]{MR2789373} and its proof, we have
\begin{equation}
\label{seqEVtoLC-e1}
f_{s}(z) = \lim_{t \to \infty} (L_{t} \circ \varphi_{s,t})(z)
\end{equation} 
for a certain family $(L_{t})_{t \ge s}$ of M\"obius transformations of $\CC$ whose coefficients consist of values of $(\varphi_{0,t})$ and its derivative $(\varphi_{0,t}')$ at the origin, where the limit \eqref{seqEVtoLC-e1} is attained uniformly on compact subsets of the unit disk $\D$.
By assumption, $L_{t}^{n}$, which appears in the relation $f_{s}^{n}(z) = \lim_{t \to \infty} (L_{t}^{n} \circ \varphi_{s,t})(z)$, converges to $L_{t}$ uniformly on $\C$. Hence
$$\lim_{n \to \infty}f_{s}^{n}(z) 
	=\lim_{n \to \infty}\lim_{t \to \infty} (L_{t}^{n} \circ \varphi_{s,t}^{n})(z) 
	=\lim_{t \to \infty} \lim_{n \to \infty}(L_{t}^{n} \circ \varphi_{s,t}^{n})(z)  = f_{s}(z)
$$
on a certain compact subset of $\D$.
\end{proof}

\subsection{The case when $\tau$ is constant}

We firstly show Theorem \ref{main01} with the additional assumption that $\tau \in \DW$ is constant.

\begin{thm}
\label{constfunction-inner}
Under the statement of Theorem \ref{main01}, additionally we assume that $\tau$ is an internal fixed point in $\D$.
Then we obtain the same consequence as Theorem \ref{main01}.
\end{thm}

\begin{proof}
Let $\rho \in (c, 1)$ with $c > 0$ and
$$
\nu[{\rho}](z) := 
\frac
{\rho\left(\frac{z-\tau}{1-\bar{\tau}z}\right) + \tau}
{1 + \bar{\tau} \rho\left(\frac{z-\tau}{1-\bar{\tau}z}\right)} \,\,:\,\, \D \xrightarrow{\textup{into}} \D
$$
Define $f_t^{\rho}(z) := f_t(\nu[{\rho}](z))$ and $g_t^{\rho}(z) := g_t(\nu[{\rho}](z))$, then
\begin{eqnarray*}
\frac
{\de_{t}f_{t}^{\rho}(z)}
{(z-\tau)(1-\closure{\tau}z)\de_{z}f_{t}^{\rho}(z)}
&=&
\frac
{\de_{t}f_{t}(\nu[\rho](z))}
{(\nu[\rho](z)-\tau)(1-\closure{\tau}\nu[\rho](z))\de_{z}f_{t}(\nu[\rho](z))}\\[3pt]
&=&
p(\nu[\rho](z), t) =: p_{\rho}(z,t) \in \HF
\end{eqnarray*}
and
\begin{eqnarray*}
\frac
{\de_{t}g_{t}^{\rho}(z)}
{(z-\tau)(\closure{\tau}z -1)\de_{z}g_{t}^{\rho}(z)}
&=&
\frac
{\de_{t}g_{t}(\nu[\rho](z))}
{(\nu[\rho](z)-\tau)(\closure{\tau}\nu[\rho](z) -1)\de_{z}g_{t}(\nu[\rho](z))}\\[3pt]
&=&
q(\nu[\rho](z), t) =: q_{\rho}(z,t) \in \HF.
\end{eqnarray*}
Therefore $(f_t^{\rho})$ and $(g_t^{\rho})$ satisfy all the assumptions in Theorem \ref{main01}.
Since  for all $t \in [0,T)$ $f_t^{\rho}$ and $g_t^{\rho}$ have continuous extensions to $\closure{\D}$, $\Phi_{\rho}$ is defined accordingly by
\begin{equation}
\label{Phirho}
\left\{
\begin{array}{lll}
\dstyle \Phi_{\rho}(z) = f_{0}^{\rho}(z), & z \in \D,\\[3pt]
\dstyle \Phi_{\rho}\left(\frac{1}{\closure{g_{t}^{\rho}(e^{i\theta})}}\right) = f_{t}^{\rho}(e^{i\theta}), &\theta \in [0,2\pi)\hspace{5pt}\textup{and}\hspace{5pt}t \in [0,T).
\end{array}
\right.
\end{equation}

Let $I_{B} \subsetneq [0,T)$ be a set such that $p_{\rho}(z,t) \in i\R$ for all $z \in \D$ and all $t \in I_{B}$, and $I_{A} := [0,T) \backslash I_{B}$.
Remark that $I_{A}$ and $I_{B}$ do not depend on $\rho$.

We prove that $\Phi_{\rho}$ is homeomorphism on $\Delta[(g_{t}^{\rho})]$.
In order to show injectivity of $\Phi_{\rho}$, we divide $\Delta[(g_{t}^{\rho})]$ into two domains;
$$
\left\{
\begin{array}{lll}
\dstyle D_{A} &:=&\dstyle \D \cup \bigcup_{t \in I_{A}}\frac{1}{\closure{g_{t}^{\rho}(\de\D)}},\\[15pt]
\dstyle D_{B} &:=&\dstyle \bigcup_{t \in I_{B}}\frac{1}{\closure{g_{t}^{\rho}(\de\D)}}.
\end{array}
\right.
$$

Take two distinct points $z_1, z_2 \in D_{A}$.
If either $z_1$ or $z_2$ is in $\D$, it s clear that $\Phi_{\rho}(z_1) \neq \Phi_{\rho}(z_2)$.
Now, let $z_1, z_2 \in D_{A}\backslash \D$.
The equation $1/\closure{z_1} = g_{t_1}^{\rho}(e^{i\theta_1})$ determines a unique $t_1 \in I_{A}$ and $\theta_1 \in [0,2\pi)$ because of the following.
It is easy to see that there does not exist another $\theta_{1}^{*} \in [0,2\pi)$ such that $1/\closure{z_1} = g_{t_1}^{\rho}(e^{i\theta_1^*})$, because the curve $\{g_{t_1}^{\rho}(e^{i \lambda}) : \lambda \in [0,2\pi)\}$ is Jordan.
If there exists another $t_{1}^{*} \in I_{A}, t_{1}^{*} > t_{1}$, such that $1/\bar{z_1} = g_{t_1^{*}}^{\rho}(e^{i\theta_1})$, then 
\begin{equation}
\label{cal_omega01}
\begin{array}{lll}
\omega_{t_1,t_1^*}(\nu[\rho](e^{i\theta_1})) 
	&= & (g_{t_1}^{-1} \circ g_{t_1^*})(\nu[\rho](e^{i\theta_1}))\\[3pt]
	&= & (g_{t_1}^{-1} \circ g_{t_1})(\nu[\rho](e^{i\theta_1}))\\[3pt]
	&= & \nu[\rho](e^{i\theta_1}).
\end{array}
\end{equation}
Since $\omega_{t_{1}, t_{1}^{*}}(\tau) = \tau$, we have 
\begin{equation*}
\eta_{\D}(\tau, \nu[\rho] (e^{i\theta_1}) ) =  \eta_{\D}(\omega_{t_{1}, t_{1}^{*}}(\tau), \omega_{t_{1}, t_{1}^{*}}(\nu[\rho] (e^{i\theta_1})) ),
\end{equation*} 
where $\eta_{\D}$ is the hyperbolic metric on $\D$. 
Hence by the Schwarz-Pick Theorem, $\omega_{t_1, t_2}$ is a conformal automorphism of $\D$.
By calculation we have
\begin{equation}
\label{rotation01}
\omega_{t_1, t}(z) = \frac
{e^{i\theta(t)}\left(\frac{z-\tau}{1-\bar{\tau}z}\right) + \tau}
{1 + \bar{\tau} e^{i\theta(t)}\left(\frac{z-\tau}{1-\bar{\tau}z}\right)} \hspace{15pt}(t \in [t_{1}, t_{1}^{*}])
\end{equation}
with some real function $\theta: [t_{1}, t_{1}^{*}] \to  \R$ with $\theta(t_{1}^{*}) = 0$. 
It describes a rotation of the unit disk $\D$ around the point $\tau$ in the hyperbolic sense.
Hence one can verify that $q(\D, t)$ lies on the imaginary axis for all $t \in [t_1, t_2)$.
It contradicts the assumption $t \in I_{A}$.

Then, suppose that there exists another $t_1^* \in I_{A},\,t_{1}^{*} \neq t_{1}$, and $\theta_1^* \in [0,2\pi),\,\theta_{1}^{*} \neq \theta_{1}$, such that $1/\bar{z_1} = g_{t_1^*}^{\rho}(e^{i\theta_1^*})$.
We may assume that $t_1^* > t_1$.
Since $g_{t_1}^{\rho}(e^{i\theta_1}) = g_{t_1^*}^{\rho}(e^{i\theta_1^*})$, we have $\omega_{t_1,t_1^*}(\nu[\rho](e^{i\theta_1^*})) = \nu[\rho](e^{i\theta_1})$ as \eqref{cal_omega01}.
Hence $\eta_{\D}(\tau, \nu[\rho] (e^{i\theta_1^{*}})) = \eta_{\D}(\omega_{t_{1}, t_{1}^{*}}(\tau), \omega_{t_{1}, t_{1}^{*}}(\nu[\rho] (e^{i\theta_1^{*}})))$.
It shows that $\omega_{t_{1}, t}$ is a rotation as \eqref{rotation01} and $p(\D, t) \in i\R$ for all $t \in [t_{1}, t_{1}^{*}]$, which again contradicts the assumption that $t \in I_{A}$.

Therefore, there exist unique $t_1, t_2 \in I_{A}$ and $\theta_1, \theta_2 \in [0, 2\pi)$ such that $z_1 = 1/\closure{g_{t_1}^{\rho}(e^{i\theta_1})}$ and $z_2 = 1/\closure{g_{t_2}^{\rho}(e^{i\theta_2})}$.
We may suppose $t_1 \le t_2$.
If $t_1 = t_2$, then since $f_{t}^{\rho}(\de \D)$ is Jordan, $\Phi_{\rho}(z_1) = f_{t_1}(e^{i\theta_1}) \neq f_{t_2}(e^{i\theta_2}) = \Phi_{\rho}(z_2)$.
In the case when $t_1 < t_2$, the same argument as above with \eqref{cal_omega01} and \eqref{rotation01} then shows that $\Phi_{\rho}(z_1) \neq \Phi_{\rho}(z_2)$.
Consequently $\Phi_{\rho}$ is injective on $D_{A}$.

We show injectivity of $\Phi_{\rho}$ on $D_{B}$.
Take one connected component $I^{*}_{B}$ in $I_{B}$.
Since $p_{\rho}(\D, t), q_{\rho}(\D, t) \in i\R$ for all $t \in I_{B}$, corresponding $\varphi_{t_{p}, t} \in \EF$ and $\omega_{t_{p}, t} \in \REF$ are rotations as \eqref{rotation01} on $t \in I^{*}_{B}$.
Hence $\{f_{t}^{\rho}(e^{i\theta}) : \theta \in [0,2\pi), t \in I^{*}_{B}\}$ and $\{g_{t}^{\rho}(e^{i\theta}) : \theta \in [0,2\pi), t \in I^{*}_{B}\}$ are curves.
For a fixed $\theta_{0} \in [0,2\pi)$, one point on the both curves are determined.
Therefore, $\Phi_{\rho} : 1/\closure{\{g_{T}^{\rho}(e^{i\theta}, t \in I^{*}_{B}) : \theta \in [0,2\pi)\}} \to \{f_{T}^{\rho}(e^{i\theta}) : \theta \in [0,2\pi), t \in I^{*}_{B}\}$ is injective.

Since $\Phi_{\rho}(D_{A}) \cap \Phi_{\rho}(D_{B}) = \emptyset$, $\Phi_{\rho}$ gives an injective map on $\Delta[(g_{t}^{\rho})]$ onto $\Omega[(f_{t}^{\rho})]$.
It concludes that $\Phi_{\rho} : \Delta[(g_{t}^{\rho})] \to \Omega[(f_{t}^{\rho})]$ is a homeomorphism.

Differentiations of \eqref{Phirho} both sides with respect to $t$ and $\theta$ yield
\begin{equation*}
\left\{
\begin{array}{lll}
\dstyle
	\de_{z}\Phi_{\rho} \cdot
	\closure{
		\left(\frac{\de_{t}g_{t}^{\rho}(\zeta )}{g_{t}^{\rho}(\zeta)^2}\right)
		} 
	+ 
	\de_{\bar{z}}\Phi_{\rho} \cdot
	\left(
		\frac{\de_{t}g_{t}^{\rho}(\zeta )}{g_{t}^{\rho}(\zeta)}
	\right)
	&=&
	-\de_{t}f_{t}^{\rho}(\zeta) \\[12pt]
\dstyle
	\de_{z}\Phi_{\rho} \cdot
	\closure{
		\left(\frac{\zeta \de_{z}g_{t}^{\rho}(\zeta)}{g_{t}^{\rho}(\zeta)^2}\right)
		}  
	- 
	\de_{\bar{z}}\Phi_{\rho} \cdot
	\left(
		\frac{\zeta \de_{z}g_{t}^{\rho}(\zeta)}{g_{t}^{\rho}(\zeta)}
	\right)
	&=&
	\zeta \de_{z}f_{t}^{\rho}(\zeta)
\end{array}
\right.
\end{equation*}
where $\zeta := e^{i\theta}$, and therefore
\begin{equation}
\label{beltrami}
\begin{array}{lll}
\dstyle\frac{\de_{\bar{z}}\Phi_{\rho}}{\de_{z}\Phi_{\rho}}\left(\frac{1}{\closure{g_{t}^{\rho}(\zeta)}}\right)
 &=& 
\dstyle\left(\frac{g_{t}^{\rho}(\zeta)^2}{\de_{t}g_{t}^{\rho}(\zeta )}\right)\closure{\left(\frac{\de_{t}g_{t}^{\rho}(\zeta )}{g_{t}^{\rho}(\zeta)^2}\right)} \cdot
\frac
{\dstyle \frac{\de_{t}f_{t}^{\rho}(\zeta )}{\zeta \de_{z}f_{t}^{\rho}(\zeta )}- \closure{\left(-\frac{\de_{t}g_{t}^{\rho}(\zeta )}{\zeta \de_{z}g_{t}^{\rho}(\zeta )}\right)}}
{\dstyle \frac{\de_{t}f_{t}^{\rho}(\zeta )}{\zeta \de_{z}f_{t}^{\rho}(\zeta )}+ \left(-\frac{\de_{t}g_{t}^{\rho}(\zeta )}{\zeta \de_{z}g_{t}^{\rho}(\zeta )}\right)}\\[30pt]
&=&
\dstyle\left(\frac{g_{t}^{\rho}(\zeta )^2}{\de_{t}g_{t}^{\rho}(\zeta )}\right)\closure{\left(\frac{\de_{t}g_{t}^{\rho}(\zeta )}{g_{t}^{\rho}(\zeta )^2}\right)} \cdot
\frac
{\dstyle
	\phi[\tau](\zeta) \,p_{\rho}(\zeta ,t) 
	- 
	\closure{
		\phi[\tau](\zeta)\,q_{\rho}(\zeta ,t)}
}
{\dstyle
	\phi[\tau](\zeta) \,p_{\rho}(\zeta ,t) 
	+ 
	\phi[\tau](\zeta) \,q_{\rho}(\zeta ,t)
}
\end{array}
\end{equation}
for almost all $t \in [0, T)$ and all $\zeta \in \de \D$, where 
$$
\phi[\tau](z) := \frac{(z -\tau)(1-\closure{\tau}z)}z.
$$
It follows from the fact $\phi[\tau](\zeta) = \closure{\phi[\tau](\zeta)}$ that $|\de_{\bar{z}}\Phi_{\rho}(z)/\de_{z}\Phi_{\rho}(z)| \leq k$ for almost all $z \in \Delta[(g_{t}^{\rho})]$, namely, $\Phi_{\rho}$ is a $k$-quasiconformal mapping on $\Delta[(g_{t}^{\rho})]$.

Take a constant $c <1$ as close enough to 1.
Since $\Phi_{\rho} = f_{0}^{\rho}$ is conformal on $\D$, one can find three distinct points $z_{1}, z_{2}, z_{3}$ as $d_{\CC}(f_{0}^{\rho}(z_{i}), f_{0}^{\rho}(z_{j}))$, $i, j = 1,2,3$, $i \neq j$, is greater than some constant $d>0$ for all $\rho \in (c,1)$, where $d_{\CC}$ stands for the spherical distance.
Hence by normality (remark that $\Delta[(g_{t}^{\rho})] = \C$ if and only if $\Delta[(g_{t})] =\C$) $\Phi$ is $k$-quasiconformal on $\Delta[(g_{t})]$.
Accordingly, for all $t \in [0,T)$ $f_{t}$ and $g_{t}$ have continuous extensions to $\closure{\D}$.
\end{proof}

\begin{thm}
\label{constfunction-boundary}
Under the statement of Theorem \ref{main01}, additionally we assume that $\tau$ is a boundary fixed point of $\D$.
Then we obtain the same consequence as Theorem \ref{main01}. 
\end{thm}

\begin{proof}
By some rotation, it is enough to consider the case when $\tau =1$.
In order to make our discussion simple, we transfer everything to the right half-plane by a conjugation with a Cayley map $C(z) = (1+z)/(1-z)$ (see \cite{HottaGumenyuk}).
For example, a family of holomorphic functions $(\phi_{s,t})_{0 \le s \le t < \infty}$ on the right half-plane $\H$ is said to be an \textit{evolution family} if $K^{-1} \circ \phi_{s,t} \circ K \in \EF$.
A corresponding Herglotz vector field $G_{\H}$ is given by the relation $(d/dt)\phi_{s,t}(z) = G_{\H}(\phi_{s,t}, t)$.
Then \eqref{HVtoBP} is written as
\begin{equation*}\label{HerglotzH}
	G_{\H}(\zeta,t)=K'(K^{-1}(\zeta))G(K^{-1}(\zeta),t)=p_\H(\zeta,t)
\end{equation*}
for all $t\ge0$ and all~$\zeta\in\H$, where $p_{\H}(\zeta,t) = 2p(K^{-1}(\zeta), t)$ stands for the right half-plane model of Herglotz function.
Accordingly \eqref{LdLDE}, \eqref{LddecLDE} and \eqref{Phi} are written as
\begin{equation*}\label{EvolutionH}
	\de_{t}f_{t}(\zeta) = -\de_{z}f_{t}(\zeta)p_{\H}(\zeta,t),\hspace{15pt}
	\de_{t}g_{t}(\zeta) = \de_{z}g_{t}(\zeta)q_{\H}(\zeta,t),
\end{equation*}
and
\begin{equation*}
\left\{
\begin{array}{lll}
\dstyle \Phi(\zeta) = f_{0}(\zeta), & \zeta \in \H,\\[3pt]
\dstyle \Phi\left(-\closure{g_{t}(iy)}\right) = f_{t}(iy), &y \in \R \hspace{5pt}\textup{and}\hspace{5pt}t \in [0,T).
\end{array}
\right.
\end{equation*}

For $\rho \in (0,1)$, define $f_{t}^{\rho}(\zeta) := f_{t}(\zeta + \rho)$ and $g_{t}^{\rho}(\zeta) := g_{t}(\zeta + \rho)$.
Then $(f_{t}^{\rho}) \in \LC$ with $(p_{\H}(\zeta + \rho,t), 1) \in \BP$ and $(g_{t}^{\rho}) \in \DLC$ with $(q_{\H}(\zeta + \rho,t), 1) \in \BP$. 
Thus $(f_{t}^{\rho})$ and $(g_{t}^{\rho})$ satisfy all the assumption of the theorem.
Further, one can see that for all $t \ge 0$ $f_{t}^{\rho}$ and $g_{t}^{\rho}$ have continuous extensions to $\closure{\D}$.
Hence
\begin{equation}\label{Phi-upper-rho}
\left\{
\begin{array}{lll}
\dstyle \Phi_{\rho}(\zeta) = f_{0}^{\rho}(\zeta), & \zeta \in \H,\\[3pt]
\dstyle \Phi_{\rho}\left(-\closure{g_{t}^{\rho}(iy)}\right) = f_{t}^{\rho}(iy), &y \in \R \hspace{5pt}\textup{and}\hspace{5pt}t \in [0,T).
\end{array}
\right.
\end{equation}
is well-defined.

In principle, we follow the similar lines of the proof of Theorem \ref{constfunction-inner} using the Julia-Wolff-Carath\`eodory Theorem (e.g. \cite[Section 1.3]{ElinShoikhet:2010}).
We again use the notations $I_{A}, I_{B} \subset [0,\infty)$ and $D_{A}, D_{B} \subset \C$ in the proof of Theorem \ref{constfunction-inner}.
In this case $D_{A}$ is understood as a union of all boundary curves of $1/\closure{g_{t}^{\rho}(\zeta)}$ for all $t \in I_{A}$ and the right half-plane $\H$.

What we will prove first is that injectivity of $\Phi_{\rho}$ on $D_{A}$.
Take two points $\zeta_{1}, \zeta_{2} \in D_{A}\backslash \H$. 
We prove that the equation $-\closure{\zeta_{1}} := g_{t_{1}}^{\rho}(iy_{1})$ determines unique $y_{1} \in \R$ and $t_{1}  \in I_{A}$.
It is easy to see that for any fixed $t_{1} \in I_{A}$, $y_{1} \in \R$ is determined uniquely because the curve $\{g_{t_{1}}^{\rho}(iy) : y \in \R\}$ is Jordan.
If there exists another $t_{1}^{*} \in I_{A},\, t_{1}^{*} > t_{1},$ such that $-\closure{\zeta_{1}} = g_{t_{1}^{*}}^{\rho}(iy_{1})$, then 
$$
\Psi_{t_{1}, t_{1}^{*}}(iy_{1} + \rho)
=(g_{t_{1}}^{-1} \circ g_{t_{1}^{*}}) (iy_{1} + \rho)
=(g_{t_{1}}^{-1} \circ g_{t_{1}}) (iy_{1} + \rho)
= iy_{1} + \rho.
$$
Therefore $\Re \Psi_{t_1,t_1^*} (\zeta_{0}) = \Re\zeta_{0}$, where $\zeta_{0} := iy_{1} + \rho$.
Hence by the Julia-Wolff-Carath\`eodory Theorem, $\Psi_{t_{1}, t_{1}^{*}}'(\infty) =1$.
Further, since the equality holds at the internal point $\zeta_{0}$, by the maximum modulus principle $\Re \Psi_{t_1,t_1^*} (\zeta) = \Re \zeta$ for all $\zeta \in \H$.
Thus $\Psi_{t_{1}, t_{1}^{*}}(\zeta) := \zeta$ and hence $p(\H, t) =0$ for all $t \in [t_{1}, t_{1}^{*}]$.
It contradicts our assumption.

Then suppose that there exists another $t_1^* \in I_{A},\,t_{1}^{*} \neq t_{1}$, and $y_{1}^* \in \R,\,y_{1}^{*} \neq y_{1}$, such that $-\closure{z_{1}} = g_{t_{1}^{*}}^{\rho}(iy_{1}^{*})$.
We may assume that $t_1^* > t_1$.
Since $g_{t_1}^{\rho}(iy_{1}) = g_{t_1^*}^{\rho}(iy_{1}^{*})$, we have $\Psi_{t_1,t_1^*}(iy_{1}^{*} + \rho) = (g_{t_1}^{-1} \circ g_{t_1^*})(iy_{1}^{*} + \rho) =  (g_{t_1}^{-1} \circ g_{t_1})(iy_{1} + \rho) = iy_{1} + \rho$.
By the same argument as above, we obtain $\Psi_{t_{1}, t_{1}^{*}}(\zeta) := \zeta + i(y_{1}^{*} -y_{1})$ and hence $p(\H, t) =0$ for all $t \in [t_{1}, t_{1}^{*}]$ which again contradicts our assumption.

Now, there exist unique $t_1, t_2 \in I_{A}$ and $y_1, y_2 \in \R$ such that $-\closure{\zeta_1} = g_{t_1}^{\rho}(iy_{1})$ and $-\closure{\zeta_2} = g_{t_2}^{\rho}(iy_{2})$.
We may suppose $t_1 \le t_2$.
If $t_{1} =t_{2}$, then since $f_{t_{1}}^{\rho}(i\R)$ is Jordan, $\Psi_{\rho}(\zeta_{1}) \neq \Psi_{\rho}(\zeta_{2})$.
Then suppose $t_{1} < t_{2}$.
If $\Psi_{\rho}(\zeta_{1}) = \Psi_{\rho}(\zeta_{2})$, i.e., $f_{t_{1}}^{\rho}(\zeta_{1}) = f_{t_{2}}^{\rho}(\zeta_{2})$, then $\Psi_{t_{1}, t_{2}}(\zeta_{1} + \rho) = (f_{t_{2}}^{-1} \circ f_{t_{1}})(\zeta_{1} + \rho) = (f_{t_{2}}^{-1} \circ f_{t_{2}})(\zeta_{2} + \rho) = \zeta_{2} + \rho$.
Therefore $\Psi_{t_{1}, t_{2}}(\zeta) = \zeta + i(y_{2}-y_{1})$, and hence $p(\H, t) =0$.
Consequently $\Phi_{\rho}$ is injective on $D_{A}$.

Injectivity of $\Phi_{\rho}$ on $D_{B}$ is obtained by the same argument as the proof of Theorem \ref{constfunction-inner}.
Hence $\Phi_{\rho} : \Delta[(g_{t})] \to \Omega[(f_{t})]$ is injective. 
We conclude that $\Phi_{\rho}$ defines a homeomorphism on $\Delta[(g_{t})]$ onto $\Omega[(f_{t})]$.

Differentiating both sides of \eqref{Phi-upper-rho} with respect to $t$ and $y$ we have
\begin{equation*}
\left\{
\begin{array}{lll}
\dstyle
	\de_{z}\Phi_{\rho} \cdot 
	(\closure{-\de_{t} g_{t}^{\rho}(iy)})
	+ 
	\de_{\bar{z}}\Phi_{\rho} \cdot
	(-\de_{t} g_{t}^{\rho}(iy))
	&=&
	\de_{t}f_{t}^{\rho}(iy) \\[3pt]
\dstyle
	\de_{z}\Phi_{\rho} \cdot
	(\closure{-i \de_{z}g_{t}^{\rho}(iy)})
	+
	\de_{\bar{z}}\Phi_{\rho} \cdot
	(-i\de_{z}g_{t}^{\rho}(iy))
	&=&
	\de_{z}f_{t}^{\rho}(iy),
\end{array}
\right.
\end{equation*}
and hence
\begin{eqnarray*}
\frac{\de_{\bar{z}}\Phi_{\rho}}{\de_{z}\Phi_{\rho}}\left(-\closure{g_{t}^{\rho}(iy)}\right)
 &=& 
\left(
	\frac{ \closure{ \de_{t}g_{t}^{\rho}(iy) } }{\de_{t}g_{t}^{\rho}(iy)}\cdot
\right)
	\frac
		{\dstyle \left(-\frac{\de_{t}f_{t}^{\rho}(iy)}{\de_{z}f_{t}^{\rho}(iy)}\right) - \closure{ \left( \frac{\de_{t}g_{t}^{\rho}(iy)}{\de_{z}g_{t}^{\rho}(iy)} \right) } }
		{\dstyle \left(-\frac{\de_{t}f_{t}^{\rho}(iy)}{\de_{z}f_{t}^{\rho}(iy)}\right) + \left( \frac{\de_{t}g_{t}^{\rho}(iy)}{\de_{z}g_{t}^{\rho}(iy)} \right) }\\
&=&
\left(
	\frac{ \closure{ \de_{t}g_{t}^{\rho}(iy) } }{\de_{t}g_{t}^{\rho}(iy)}\cdot
\right)
	\frac
		{ p_{\H}(iy + \rho, t) - \closure{ q_{\H}(iy+\rho, t) } }
		{ p_{\H}(iy + \rho, t) + q_{\H}(iy+\rho, t) }.
\end{eqnarray*}

\no
Thus $|\de_{\bar{z}}\Phi_{\rho}(\zeta)/\de_{z}\Phi_{\rho}(\zeta)| \le k$ for almost all $\zeta \in \Delta[(g_{t}^{\rho})]$ and hence $\Phi_{\rho}$ is $k$-quasiconformal there.
By the same argument as in the proof of Theorem \ref{constfunction-inner}, we conclude that $\Phi : \Delta[(g_{t})] \to \Omega[(f_{t})]$ is $k$-quasiconformal. The proof is complete.
\end{proof}

\subsection{The case when $\tau$ is a step function}
\label{stepsubsection}

Next, we assume that $\tau \in \DW$ is a step function, that is, $\tau$ is of the form
$$
\tau (t) = \sum_{i=1}^{n}\tau_{i}\cdot\chi_{I_{i}}(t),
$$
where $\tau_{i} \in \closure{\D},\,n \in \N$, $0=t_0 < t_1  < \cdots < t_{n-1} <  t_{n} = T$, $I_i := [t_{i-1}, t_i)$ and $\chi_{I}$ is a characteristic function.

\begin{thm}
\label{stepfunction}
Under the statement of Theorem \ref{main01}, additionally we assume that $\tau$ is a step function (in the above sense).
Then we obtain the same consequence as Theorem \ref{main01}.
\end{thm}

\begin{proof}
With no loss of generality, we may assume that $I_{i} \backslash I_{B}$ is not a null-set for all $i = 1, \cdots, n$.

Firstly consider the case when $t \in I_{1}$.
Then applying the same argument as Theorem \ref{constfunction-inner} if $\tau_{1} \in \D$ or Theorem \ref{constfunction-boundary} if $\tau_{1} \in \de\D$ for $(f_{t})$ and $(g_{t})$ on $t \in I_{1}$, one can verify that the map $\Phi_{1}$ defined by
$$
\left\{
\begin{array}{lll}
\dstyle \Phi_1(z) = f_{0}(z), & z \in \D,\\[3pt]
\dstyle \Phi_1\left(1/\closure{g_{t}(e^{i\theta})}\right) = f_{t}(e^{i\theta}), &\theta \in [0,2\pi)\hspace{5pt}\textup{and}\hspace{5pt}t \in I_1,
\end{array}
\right.
$$
is a $k$-quasiconformal mapping on $\{1/\closure{w} : w \in \CC\backslash \closure{g_{t_1}(\D)}\}$ maps to $f_{t_1}(\D)$. 
We remark that the above mapping gives a quasiconformal extension of $g_{t_{1}}$ to the unit disk $\D$ (and hence $\C$).
Therefore $g_{t_{1}}$ has a continuous injective extension to $\closure{\D}$.

Next, let $t \in I_{2}$.
Then again by means of the same argument as Theorem \ref{constfunction-inner} or Theorem \ref{constfunction-boundary}, $\Phi_{2}$ given by
$$
\Phi_2\left(1/\closure{g_{t}(e^{i\theta})}\right) = f_{t}(e^{i\theta}), \hspace{10pt}\theta \in [0,2\pi)\hspace{10pt}\textup{and}\hspace{10pt}t \in I_2,
$$
define a $k$-quasiconformal mapping on $\{1/\closure{w} : w \in g_{t_2}(\D) \backslash \closure{g_{t_{1}}(\D)}\}$ maps to $f_{t_2}(\D)\backslash \closure{f_{t_{1}}(\D)}$. 
Further, since $\Phi_{2}$ gives a quasiconformal extension of $f_{t_{1}}$ to $f_{t_{2}}(\D)$, $f_{t_{1}}$ has a continuous injective extension to $\closure{\D}$.
Hence the well-defined mapping $\Phi$,
$$
\Phi(z) := \left\{
\begin{array}{lll}
\dstyle \Phi_1(z), & z \in \{1/\closure{w} : w \in \CC \backslash \closure{g_{t_{1}}(\D)}\},\\[3pt]
\dstyle \Phi_2(z), & z \in \{1/\closure{w} : w \in g_{t_2}(\D) \backslash \closure{g_{t_{1}}(\D)}\},
\end{array}
\right.
$$
is a homeomorphism of $\{1/\closure{w} : w \in \C \backslash g_{t_{2}}(\D)\}$ onto $f_{2}(\D)$ and is $k$-quasiconformal there.
Repeating this argument, our proof is complete.
\end{proof}

\subsection{The case when $\tau$ is a measurable function}
\label{generalcase}

Now we are ready to prove Theorem \ref{main01}.

\begin{proof}[\textbf{Proof of Theorem \ref{main01}}]
We may assume that $(f_{t}) \in \LC_{0}$, i.e., range-normalized.
Since $\tau$ is measurable, there exists a sequence of step functions $\{\tau_{n}\} \subset \DW$ converging weakly to $\tau$.
For each $n \in \N$, $(p, \tau_{n}) \in \BP$ associates unique $G_{n} \in \HV$ and $(\varphi_{s,t}^{n}) \in \EF$, and $(q, \tau_{n})$ associates unique $G^{*}_{n} \in \HV$ and $(\omega_{s,t}^{n}) \in \REF$.
Let $(f_{t}^{n}) \in \LC_{0}$ being uniquely associated with $(\varphi_{s,t}^{n}) \in \EF$ and $(g_t^{n}) :=(\omega_{0,t}^{n}) \in \DLC$.
By Theorem \ref{stepfunction}, $\Phi_{n}$ defined by 
\begin{equation}\label{Phi_n}
\left\{
\begin{array}{lll}
\dstyle \Phi_{n}(z) = f_{0}^{n}(z), & z \in \D,\\[3pt]
\dstyle \Phi_{n}\left(\frac{1}{\closure{g_{t}^{n}(\zeta)}}\right) = f_{t}^{n}(\zeta), &\zeta \in \de\D\hspace{5pt}\textup{and}\hspace{5pt}t \in [0,T),
\end{array}
\right.
\end{equation}
is a $k$-quasiconformal map on $\Delta[(g_{t}^{n})]$.

By assumption, $p(\D,t)$ is bounded for almost all $t \in [0,T]$.
Hence $\{G_{n}(\,\cdot\,, t) : n \in \N, \mathrm{a.e.}~t \in [0,T]\}$ is normal.
Further, $\{G_{n}(z,\,\cdot\,)\}_{n \in \N}$ converges weakly to $G(z,\,\cdot\,)$ for all fixed $z \in \D$.
Applying Lemma \ref{HVweaktoEF}, we have $(\varphi_{s,t}^{n}) \to (\varphi_{s,t})$ locally uniformly on $(z,t) \in \D \times [s,\infty)$.
By Lemma \ref{seqEVtoLC}, $(f_{t}^{n}) \to (f_{t})$ locally uniformly on $z \in \D$ for each $t \in [0,T)$.
The same argument is applied for the decreasing counterpart, i.e., $(\omega_{s,t}^{n}) \to (\omega_{s,t})$ locally uniformly on $(z,t) \in \D \times [s,\infty)$. In particular $(g_{t}^{n})\to (g_{t})$ locally uniformly on $(z,t) \in \D \times [0, T)$.

Now, take an arbitrary $t \in [0, T)$ and fixed.
Let $D_{n} := \{1/\closure{w} : w \in \CC\backslash \closure{g_{t}^{n}(\D)}\}$,
and consider conformal maps $h_{n} : \D \to D_{n}$ and $k$-quasiconformal maps $\Psi_{n} :=\Phi_{n}\!\!\mid_{D_{n}}$.
Since $g_{t}^{n} \to g_{t}$ locally uniformly on $\D$ and $g_{t}$ is conformal on $\D$,
there exists $N \in\N$ and an open set $K$ such that $K \subset g_{t}^{n}(\D)$ for all $n > N$.
Since $D_{n} \cap \{1/\bar{z} : z \in K\} = \emptyset$ for all $n >N$, by Montel's theorem $\{h_{n}\}_{n}$ forms a normal family.
Then one can find a subsequence $\{h_{n_{j}}\}_{j} \subset \{h_{n}\}$ converging locally uniformly to a conformal map $h :\D \to D := \{1/\closure{w} : w \in \CC\backslash \closure{g_{t}(\D)}\}$.
Applying a similar argument to $\{\Psi_{n_{j}}\circ h_{n_{j}}\}_{j}$, it is also normal, and hence there is a subsequence (we keep the same notation for the subsequence) converging to a $k$-quasiconformal map on $\D$ (see \cite[Theorem II-5.2]{LehtoVirtanen:1973}).
It shows that the limit function $\Psi := \lim_{j\to \infty}\Psi_{n_{j}}$ is defined on $D$, equal to $\Phi\!\!\mid_{D}$ and $k$-quasiconformal there.

Since $t$ is arbitrary, the first assertion is proved.
It also concludes that $f_{t}$ and $g_{t}$ have continuous extensions to $\closure{\D}$ for all $t \in [0, T)$.
\end{proof}

	%
\section{Proof of Theorem \ref{openingthm} and further results}\label{Section4}

\subsection{Proof of Theorem \ref{openingthm}}
Various corollaries are deduced from the results and the arguments in Section 3.

\begin{lem}\label{(1,tau=constant)}
Let $\tau \in \closure{\D}$ be a constant. Then we have the followings;
\begin{enumerate}
\item For $(f_{t}) \in \LC$ being associated with $(1,\tau) \in \BP$, $\Omega[(f_{t})] = \C$.
\item For $(g_{t}) \in \DLC$ being associated with $(1,\tau) \in \BP$, $\Lambda[(g_{t})]$ consists of a single point in $\closure{\D}$.
\end{enumerate}
\end{lem}

\begin{proof}
$(f_{t}^{\tau})$ and $(g_{t}^{\tau})$ associated with $(1, \tau)$ is explicitly given by
$$
f_{t}^{\tau} (z):=
\frac
{\dstyle\left(\frac{z-\tau}{1-\bar{\tau}z}\right)e^{t(1-|\tau|^{2})} + \tau}
{\dstyle 1 +\bar{\tau}\left(\frac{z-\tau}{1-\bar{\tau}z}\right)  e^{t(1-|\tau|^{2})}}, \hspace{15pt}
g_{t}^{\tau} (z):= f_{-t}^{\tau} (z) 
$$
if $\tau \in \D$, and
$$
f_{t}^{\tau} (z):=
\frac{\tau +z}{\tau-z} -2t, \hspace{15pt}
g_{t}^{\tau} (z):=
\tau \cdot
\frac
{\dstyle f_{-t}^{\tau} (z) -1}
{\dstyle f_{-t}^{\tau} (z) +1}
$$
if $\tau \in \de\D$.
It is easy to see that in both cases our assertions are satisfied.
\end{proof}

\begin{rem}
\label{remark-LR}
The same conclusion of Lemma \ref{(1,tau=constant)} is also true if $\tau : [0,\infty) \to \closure{\D}$ is a step function.
\end{rem}

\begin{lem}\label{(1,tau)}
We have the followings;
\begin{enumerate}
\item For $(f_{t}) \in \LC$ being associated with $(1,\tau) \in \BP$, $\Omega[(f_{t})] = \C$.
\item For $(g_{t}) \in \DLC$ being associated with $(1,\tau) \in \BP$, $\Lambda[(g_{t})]$ consists of a single point in $\closure{\D}$.
\end{enumerate}
\end{lem}

\begin{proof}
We may assume that $(f_{t}) \in \LC_{0}$.
Let $\{\tau_{n}\}_{n \in \N}$ be a sequence of step functions converging weakly to $\tau$.
Let $(f_{t}^{n}) \in \LC_{0}$ and $(g_{t}^{n}) \in \DLC$ associated with $(1, \tau_{n}) \in \BP$.
By Remark \ref{remark-LR}, $\Omega[(f_{t}^{n})] = \C$ and $\Lambda[(g_{t}^{n})]$ consists of a single point in $\closure{\D}$ for all $n \in \N$.
By Theorem \ref{stepfunction}, one can deduce that $\Phi_{n}$, a map defined by $(f_{t}^{n})$ and $(g_{t}^{n})$ as in Section \ref{stepsubsection}, is a 0-quasiconformal automorphism of $\CC$, i.e., a M\"obius transformation.
Since $f_{0} \in \SS$, $\Phi^{n} = \mathsf{id}_{\CC}$ for all $n \in \N$.
Hence a local uniform limit $\Phi$ of $\{\Phi_{n}\}$ is $\mathsf{id}_{\C}$.
This proves that $f_{t}(\D)$ tends to $\C$ as $t \to \infty$. 
Accordingly, $\Lambda[(g_{t})]$ consists of a single point in $\closure{\D}$.
\end{proof}

\begin{proof}[\textbf{Proof of Theorem \ref{openingthm}}]
Proof of (i): Consider the decreasing Loewner chain $(g_{t})$ as the one associated with the Berkson-Porta data $(1, \tau) \in \BP$.
Applying Theorem \ref{main01}, we immediately obtain the first assertion of the theorem.

Proof of (ii): We employ the same idea as the proof of \cite[Theorem 3.5]{HottaGumenyuk}. 
Let us fix $s \ge 0$ and $t \ge s$.
Then one defines $(\tilde{f}_{a})_{a \ge 0}$ by
\begin{equation*}
\tilde{f}_{t}(z) := \left\{
\begin{array}{llll}
\varphi_{s+a, t}(z), & a \in [0,t),\\
e^{a-t} z, & a \in [t, \infty).
\end{array}
\right.
\end{equation*}
By the definition, $(\tilde{f}_{a})_{a \ge 0}$ is a Loewner chain.
Since $\varphi_{s+a, t} = f_{t}^{-1} \circ f_{s+a}$, $\tilde{f}_{t}$ satisfies
$$
\frac{\de \tilde{f}_{a}(z)}{\de a} = \frac{\de \tilde{f}_{a}(z)}{\de z}(z - \tau(s+a))(1-\closure{\tau(s+a)} z) p(z, s+a)
$$
of all $z \in \D$ and almost all $a \in [0,t)$, a Herglotz function $\tilde{p}$ associated with $(\tilde{f}_{a})$ is given by
\begin{equation*}
\tilde{p}(z,a) = \left\{
\begin{array}{llll}
p(z,a+s), & a \in [0,t),\\
1, & a \in [t, \infty).
\end{array}
\right.
\end{equation*}
By the first assertion of this theorem, we conclude that $\varphi_{s,t}$ has a $k$-quasiconformal extension to $\CC$ for each $s \ge 0$ and $t \ge s$.

Proof of (iii): By Lemma \ref{(1,tau)}, $\Delta[(g_{t})] = \CC \backslash\{a\}$ where $\{a\} := \Lambda[(g_{t})]$. 
Since $\Phi : \Delta[(g_{t})] \to \Omega[(f_{t})]$ in \eqref{Phi} is a homeomorphism, it follows that $\Omega[(f_{t})] = \C$.
\end{proof}

If $p(z,t) =q(z,t)$ for all $z \in \D$ and almost all $t \in [0,\infty)$, then inequality \eqref{maininequality} represents a sector domain.

\begin{cor}
Let $k \in [0, 1)$.
Let $(f_t) \in \LC$, $(p,\tau) \in \BP$ associated with $(f_t)$, and $(g_{t}) \in \DLC$ associated with $(p,\tau)$.
We denote by $T \in [0,\infty]$ the smallest number such that $p(\D,t)  \neq 0$ for almost all $t \in (T, \infty)$.
Suppose that 
\begin{enumerate}
\item $T > 0$;
\item There exists a compact subset of $\H$ that contains $p(\D, t)$ for almost all $t \in [0,T)$;
\item $p$ satisfies
\begin{equation*}
p(z,t) \in \closure{\left\{z : |\arg z| < \frac{k\pi }{2} \right\}}
\end{equation*}
for all $z \in \D$ and almost all $t \in [0,\infty)$.
\end{enumerate}
Then, for each $t \in [0,T)$ $f_{t}$ has a $\sin (k\pi /2)$-quasiconformal extension to $\Delta[(g_{t})]$ onto $\Omega[(f_{t})]$, where $(g_{t}) \in \DLC$ associated with $(p,\tau) \in \BP$.
\end{cor}

Choosing $(f_{t})$ as $p \equiv 1$, we obtain the following.
\begin{cor}
Let $(g_t) \in \DLC$ and $(q,\tau) \in \BP$ associated with $(g_t)$.
If $q$ satisfies
\begin{equation*}
\left|
\frac{q(z,t)-1}{q(z,t)+1}
\right| \leq k
\end{equation*}
for all $z \in \D$ and almost all $t \in [0,\infty)$. Then:
\begin{enumerate}
\item[(i)] for each $t \in [0,\infty)$, $g_{t}$ has a $k$-quasiconformal extension to $\CC$;
\item[(ii)] for each $s \in [0,\infty)$ and $t \in [s,\infty)$, $\omega_{s,t}:= g_{s}^{-1} \circ g_{t}$ has a $k$-quasiconformal extension to $\CC$;
\item[(iii)] $\Lambda[(g_{t})]$ consists of a single point in $\closure{\D}$.
\end{enumerate}
\end{cor}

\subsection{Loewner Range}
We close the paper with some considerations on the Loewner range $\Omega[(f_{t})]$.
Under some restriction on $\tau$, one can obtain the same conclusion as Theorem \ref{openingthm} (iii) with a weaker assumption on the Herglotz function $p$.

\begin{lem}
\label{loewnerrange}
Let $(f_t) \in \LC$ and $(p, \tau) \in \BP$ associated with $(f_t)$.
Suppose that $\tau \in \closure{\D}$ is a constant, and there exist uniform constants $C_{1}, C_{2} > 0$ such that 
\begin{equation}
\label{loewnerrange-condition}
C_{1}< \Re p(z,t) < C_{2}
\end{equation}
for all $z \in \D$ and almost all $t \in [0,\infty)$.
Then $\Omega[(f_t)] = \C$.
\end{lem}

\begin{proof}
Suppose firstly $\tau \in \DW$ is an internal fixed point of $\D$.
In this case we may assume that $\tau = 0$.
Let $(\varphi_{s,t}) \in \EF$ associated with $(f_t)$.
Then in view of $f_t{'}(0) = 1/\varphi_{0,t}'(0)$, Koebe's 1/4-Theorem shows that $f_t(\D)$ contains a disk radius $1/(4|\varphi_{0,t} '(0)|)$.
Hence what we need to prove is $\lim_{t \to \infty}|\varphi_{0,t}'(0)| = 0$.
Since $\tau=0$, $(\varphi_{s,t})$ satisfies
$$
\dot{\varphi}_{s,t}(z) = -\varphi_{s,t}(z) p(\varphi_{s,t}(z),t).
$$
Then calculations show that
\begin{eqnarray*}
&&
\frac{\dot{\varphi}_{0,t}(z)}{\varphi_{0,t}(z)} =- p(\varphi_{0,t}(z),t)\\
&\Longrightarrow&
\log \frac{\varphi_{0,t}(z)}{z} =- \int_0^t p(\varphi_{0,u}(z),u) du\\
&\Longrightarrow&
\Re\log \varphi'_{0,t}(0) = -\int_{0}^t \Re p(0,u)du  < -C_1 t.\\
\end{eqnarray*}
It shows that $|\varphi_{0,t}'(0)| \to 0$ as $t \to \infty$.

The case when $\tau \in \DW$ is a boundary fixed point of $\D$ is verified in \cite[Proposition 3.7]{HottaGumenyuk}.
\end{proof}

A direct corollary of Lemma \ref{loewnerrange} is that the same conclusion holds if $\tau$ is a step function in $\closure{\D}$.
Hence, one may expect that condition \eqref{loewnerrange-condition} is sufficient to obtain $\Omega[(f_{t})] = \C$ for any $\tau \in \DW$.
However in \cite{GumenyukPrause:2018}, an example of $(p, \tau) \in \BP$ is provided that $p$ satisfies \eqref{loewnerrange-condition} but $\Omega[(f_{t})] \neq \C$ where $(f_{t}) \in \LC$ associated with the $(p, \tau)$.

\bibliographystyle{amsalpha}
\def\cprime{$'$}
\providecommand{\bysame}{\leavevmode\hbox to3em{\hrulefill}\thinspace}
\providecommand{\MR}{\relax\ifhmode\unskip\space\fi MR }
\providecommand{\MRhref}[2]{%
  \href{http://www.ams.org/mathscinet-getitem?mr=#1}{#2}
}
\providecommand{\href}[2]{#2}

\end{document}